\documentclass{article}
\usepackage{authblk}
\usepackage{etex}
\usepackage[utf8]{inputenc}
\usepackage[english]{babel}
\usepackage{amsmath}
\usepackage{amsthm}
\usepackage{amssymb}
\usepackage{sectsty}
\usepackage{titlesec}
\usepackage{color}
\usepackage[color,matrix,arrow]{xy}
\usepackage{amsgen}
\usepackage{amstext}
\usepackage{amsbsy}
\usepackage{amsopn}
\usepackage{amsfonts}
\usepackage{eepic}
\usepackage{graphicx}
\usepackage{epsf}
\usepackage{pstricks}
\usepackage{float}
\usepackage{enumerate}
\usepackage{eqnarray}
\usepackage{faktor}
\xyoption{all}

\usepackage{pgf}
\usepackage{tikz-cd}
\usetikzlibrary{automata, arrows.meta, positioning,decorations.pathmorphing}

\newcommand{\footrecall}[1]{%
} 
\usetikzlibrary{snakes,shapes,arrows,automata}

\titleformat*{\section}{\large\bfseries}
\titleformat*{\subsection}{\normalsize \bfseries}

\newcommand{\inv}{^{-1}}
\newcommand{\N}{\mathbb{N}}
\newcommand{\Z}{\mathbb{Z}}
\newcommand{\oD}{D_{\sup}}
\newcommand{\oDc}{D_{\sup}^*}

\newcommand{\sss}{\Sigma^*}

\newcommand{\Rec}{\text{Rec}}
\newcommand{\Rat}{\text{Rat}}

\newcommand{\Red}{\text{Red}}
\newcommand{\Sigmap}{\Sigma^{<p}}

\newcommand{\mc}{\mathcal}

\theoremstyle{definition}
\newtheorem{theorem}{Theorem}[section]
\newtheorem{corollary}[theorem]{Corollary}

\newtheorem{proposition}[theorem]{Proposition}
\newtheorem{question}[theorem]{Question}

\newtheorem{lemma}[theorem]{Lemma}

\newtheorem{remark}[theorem]{Remark}

\begin{document}
 
 
\title{A language-theoretic approach to study the density of subsets in free groups}

\author{André Carvalho}
\affil{Centro de Investigação em Matemática e Aplicações (CIMA)
	
	Departamento de Matemática, Escola de Ciências e Tecnologia da Universidade de Évora
	
	Rua Romão Ramalho, 59, 7000–671 Évora, Portugal
	
	\texttt{andre.carvalho@uevora.pt}\\}

\maketitle

\begin{abstract}
In this paper, we study the density of subsets of nonabelian free groups using relative densities of languages. We start by proving some basic properties about the density of a language $L_1$ relative to another language $L_2$ containing $L_1$. We then focus on the case where $L_2$ is the language  of freely reduced words over an alphabet and prove an analogue of the Infinite Monkey Theorem for this language. This result, obtained as a corollary of a broader theorem on irreducible subshifts of finite type, allows for a language-theoretic characterization of rational subsets with positive density. As a consequence, we obtain a language-theoretic proof that the automorphic orbit of an element of a nonabelian free group has natural density zero, which generalizes a result by Burillo and Ventura concerning the density of primitive elements of free groups. We then describe rational subsets of free groups of positive density and explore in depth the case of finitely generated subgroups. We prove that if the rational subset is a subgroup, then it has positive density if and only if it has finite index and characterize those for which there is convergence. In cases where convergence fails due to parity constraints, we show that the density of the subset always exhibits weak convergence and that the average of its  supremum and infimum densities converges to the expected value.
\end{abstract}

\section{Introduction}

Informally, the classical infinite monkey theorem states that a monkey typing randomly in a typewriter for an infinite amount of time will, with probability $1$, type any given string. Formally, for an alphabet $\Sigma$, defining the density of a language $L\subseteq \sss$ as $D(L)=\lim_{n\to \infty}\frac{|\Sigma^n\cap L|}{|\Sigma^n|}$, the theorem states that $D(\sss s \sss)=1$, for all $s\in \sss$. It is remarked in \cite{[Sin16]} that other definitions of density have been used in the past. For example, Berstel \cite{[Ber73]} and Salomaa-Soittola \cite{[SS78]} work with $D^*(L)=\lim_{n\to \infty}\frac{\sum_{i=0}^{n}|L\cap \Sigma^i|}{\sum_{i=0}^{n}|\Sigma^i|}$, but that these are equivalent, in the sense that $D^*(L)$ exists if and only if $D(L)$ exists and, in case they do, they are equal by the Stolz-Cesáro theorem and its partial converse \cite[Theorem 1.22 and 1.23]{[Mur09]}.

Sin'ya proved a converse of the infinite monkey theorem for regular languages, showing that a regular language $L$ has density $0$ if and only if there is some forbidden subword, that is, one that does not appear as a subword of any word in $L$ \cite{[Sin15],[Sin16]}. We remark that the converse of the infinite monkey theorem does not hold in general, as it suffices to take, for example, the language of palindromes, which has density $0$ but no forbidden subword.
While this is well-understood for the free monoid, the structure of the free group $F_k$ introduces local constraints given by cancellation.

In \cite{[BV02]}, the authors define the \emph{natural density} of a subset $S\subseteq G$ of a finitely generated group $G$ with finite generating set $X$ as $$\delta_X(S)=\limsup_{n\to \infty}\frac{|S\cap B_X(n)|}{|B_X(n)|},$$ where $B_X(n)$ represents the ball of radius $n$ (centered at the identity) with respect to $X$, and they show that the set of primitive elements of a free group has natural density $0$. Under some  assumptions on the growth of a group $G$, they also prove that subgroups $H$ of finite index have density $\frac{1}{[G:H]}$ and that, in that case, the limit always exists. However, free groups do not satisfy the hypotheses of the theorem and the authors present an example of a subgroup of index $2$ such that $\delta_X(H)=\frac 3 4$ and the $\liminf$ version of density yields $\frac 1 4$ (and so the limit may not exist).

In this paper, we explore the notion of relative density for languages $L_1 \subseteq L_2$, which measures the likelihood of a word of $L_2$ belonging to $L_1$, in the case where $L_2$ is the language of reduced words over an alphabet. This will provide us with a language-theoretical framework to study the natural density of subsets of the free group. The notion of relative densities of languages has been explored in the past in different contexts \cite{[Ber73],[Koz05]}.

We start by showing some basic properties on relative densities and prove a result on the densities of languages of subshifts of finite type, which will yield as a corollary an infinite monkey theorem holds for reduced words, meaning, informally, that a monkey typing reduced words for an infinite amount of time will eventually type any given reduced word with probability $1$. Formally, given two languages $L_1\subseteq L_2\subseteq \sss$, we define the density of $L_1$ relative to $L_2$ as 
$$D(L_1\mid L_2)=\lim_{n\to \infty}\frac{|\Sigma^n\cap L_1|}{|\Sigma^n\cap L_2|},$$
when it exists, and 
we denote by $\Red(X)$ the language of freely reduced words over $\Sigma=X\cup X\inv$.
We can also define $D_{\inf}$, $D_{\sup}$ to be the same notions, but replacing $\lim$ by $\liminf$ and $\limsup$, respectively. Also, as in the previous case, we write $D^*$ when the density is cumulative, that is, calculated on balls instead of spheres.

\newtheorem*{infinitemonkey reduced}{Corollary \ref{infinitemonkey reduced}}
\begin{infinitemonkey reduced}
	Let $X=\{x_1,\ldots, x_k\}$ and $L\subseteq \Red(X)$.  If $\sss s\sss\cap L=\emptyset$ for some $s\in \Red(X)$, then $D(L\mid \Red(X))=0$.
\end{infinitemonkey reduced}

Using the fact that every element of a nonabelian free group admit the existence of orbit-blocking words, that is, words that cannot appear as subwords of cyclically reduced elements in its automorphic orbit \cite{[KOOS26]}, we get a (simple) language-theoretic proof that automorphic orbits of elements have density $0$, generalizing Burillo-Ventura's result on primitive elements.

\newtheorem*{cor: shpilrain}{Corollary \ref{cor: shpilrain}}
\begin{cor: shpilrain}
	Let $g$ be an element of a nonabelian free group $F_k$ and $S'$ be the set of   images of $g$ under an automorphism of $F_k$.  Then $\delta(S')=0$.
\end{cor: shpilrain}

We then prove an analogue of Sin'ya's Theorem for rational languges of reduced words, that is, that the converse of the infinite monkey theorem holds for rational subsets of the free group:
\newtheorem*{sinya analogue}{Theorem \ref{sinya analogue}}
\begin{sinya analogue}
	Let $L\subseteq \Red(X)$ be a rational language and let $N=\{w\in \Red(X)\mid \sss w \sss \cap L \neq\emptyset\}$. Then,  the following are equivalent:
	\begin{enumerate}
		\item $\oDc(L\mid \Red(X))=0$;
		\item $\oD(L\mid\Red(X))=0;$
		\item  $D(L\mid \Red(X))=0;$
		\item  $D(N\mid \Red(X))=0;$
		\item  $\exists w\in \Red(X) : \sss w \sss \cap L=\emptyset.$
	\end{enumerate}
\end{sinya analogue}

Using this, we are able to give a language-theoretic approach to describe rational subsets of the free group with positive natural density and, in the particular case where the rational subset is a finitely generated subgroup, then it must be of finite index.

\newtheorem*{2cosets}{Theorem \ref{2cosets}}
\begin{2cosets}
	Let $S\subseteq F_X$ be a rational subset of a nonabelian free group.  Then $\delta(S)>0$ if and only if $F_X$ is there is some $k\in\N$ and finitely many elements $a_i,b_i\in F_X$, with $i\in [k]$ such that 
	$F_X=\bigcup_{i,j\in [k]} a_iSb_j$.
\end{2cosets}

\newtheorem*{cor: fi}{Corollary \ref{cor: fi}}
\begin{cor: fi}
	Let $H\leq_{f.g.} F_X$ be a finitely generated subgroup of a nonabelian free group.  Then $\delta(H)>0$ if and only if $H$ has finite index.
\end{cor: fi}

While it is known that random walks on groups measure index uniformly under an aperiodicity assumption \cite[Theorem 1.11]{[Toi20]}, the previous result does not hold when we change the notion of density replacing $\limsup$ by $\lim$ as, as mentioned above, there are finite index subgroups of the free groups for which we don't have convergence. We are able to characterize those and show that, although there is no convergence, the average of the $\limsup$ density with the $\liminf$ density is $\frac{1}{[F_n:H]}$. To do so, we consider non-backtracking random walks on the Stallings graph of $H$ and show that, if the subgroup has a word of odd length, then the walks are aperiodic and the limit exists and is $\frac{1}{[F_n:H]}$; if not, the walks have period $2$ since, in this case,  the Stallings graph is bipartite. In that case, the limit does not exist, but we have convergence \emph{on average}.

\newtheorem*{even bipartite}{Theorem \ref{average finite index}}
\begin{even bipartite}
	Let $L$ be the language of reduced words representing elements in a subgroup $H$ of a nonabelian free group $F_n$. Then $$\frac{\oD(L\mid\Red(X))+D_{\inf}(L\mid \Red(X))}{2}=\frac{1}{[F_n:H]}.$$ Moreover, $D(L\mid \Red(X))$ exists if and only if $H$ has infinite index or has at least a word of odd length.
\end{even bipartite}

Finally, we define the weak density of a subset $K$ of a group as $G$ as 
$$\Delta(K)=\lim_{n\to \infty}\frac{1}{n}\sum_{i=0}^{n-1}\frac{|B(i)\cap K|}{|B(i)|},$$ following \cite{[BPR09]}. We are then able to prove that convergence in this week sense is well-behaved for subgroups of the free group.

\newtheorem*{weak density}{Theorem \ref{weak density}}
\begin{weak density}
	Let $H$ be a subgroup of a nonabelian free group $F_k$. Then 
	\begin{align*}
		\Delta(H) = 
		\begin{cases}
			0               & \text{if } H \text{ has infinite index} \\
			\frac{1}{[F_k:H]} & \text{if } H \text{ has finite index}
		\end{cases}.
	\end{align*}
\end{weak density}

\section{Preliminaries}
Let $G=\langle A\rangle$ be a finitely generated group, $A$ be a finite generating set, $\Sigma=A\cup A^{-1}$ and $\pi:\Sigma^*\to G$ be the canonical (surjective) homomorphism. 

A subset $K\subseteq G$ is said to be \emph{rational} if there is some rational language $L\subseteq \Sigma^*$ such that $L\pi=K$ and \emph{recognizable} if $K\pi^{-1}$ is rational. 

We will denote by $\Rat(G)$ and $\Rec(G)$ the class of rational and recognizable subsets of $G$, respectively. Rational subsets generalize the notion of finitely generated subgroups.

\begin{theorem}[\cite{[Ber79]}, Theorem III.2.7]
	\label{AnisimovSeifert}
	Let $H$ be a subgroup of a group $G$. Then $H\in \Rat(G)$ if and only if $H$ is finitely generated.
\end{theorem}

Similarly, recognizable subsets generalize the notion of finite index subgroups.

\begin{proposition}
	\label{rec fi}
	Let $H$ be a subgroup of a group $G$. Then $H\in \Rec(G)$ if and only if $H$ has finite index in $G$.
\end{proposition}

In fact, if $G$ is a group and $K$ is a subset of $G$ then $K$ is recognizable if and only if $K$ is a (finite) union of cosets of a subgroup of finite index.

In case the group $G$ is a free group with basis $A$ with surjective homomorphism $\pi:\Sigma^*\to G$, given  $L\subseteq \Sigma^*$, we define the set of reduced words representing elements in $L\pi$ by $$\overline L=\{w\in \Sigma^*\mid w \text{ is reduced and there exists $u\in L$ such that $u\pi=w\pi$} \}.$$ 
Benois' Theorem provides us with a useful characterization of rational subsets in terms of reduced words representing the elements in the subset.

\begin{theorem}\cite[Benois]{[Ben79]}
	Let $F$ be a finitely generated free group with basis $A$ and let $L \subseteq \Sigma^*$. Then $\overline{L}$ is a rational language of $\Sigma^*$ if and only if $L\pi$ is a rational subset of $F$. 
\end{theorem}

Let $\Sigma$ be an alphabet. We define the density of a language $L\subseteq \Sigma^*$ as $$D(L)=\lim_{n\to \infty} \frac{|L\cap \Sigma^n|}{|\Sigma^n|},$$ when it exists. Analogously, we define $D_{\sup}(L)=\limsup_{n\to \infty} \frac{|L\cap \Sigma^n|}{|\Sigma^n|}$  and $D_{\inf}(L)=\liminf_{n\to \infty} \frac{|L\cap \Sigma^n|}{|\Sigma^n|}$.
It follows easily from the definitions that:
\begin{enumerate}
\item if $L_1\subseteq L_2$ and $D(L_1)=1$, then $D(L_2)=1$;
\item if $D(L)$ exists, then $D(\sss \setminus L)=1- D(L)$.
\end{enumerate}

As mentioned in the introduction, using \cite[Theorems 1.22 and 1.23]{[Mur09]}, it can be seen that this coincides with the following definition:
$$D^*(L)=\lim_{n\to \infty}\frac{\sum_{i=0}^{n}|L\cap \Sigma^i|}{\sum_{i=0}^{n}|\Sigma^i|}.$$

The infinite monkey theorem states that, for any $s\in \sss$, we have $D(\sss s \sss)=1$. Using the properties above, the following theorem follows easily from the infinite monkey theorem.
\begin{theorem}
Let $L\subseteq \sss$. If $\sss s\sss\cap L=\emptyset$ for some $s\in \sss$, then $D(L)=0$.
\end{theorem} 
The converse of this theorem does not hold in general, but it does if $L$ is rational (see \cite{[Sin16], [Sin15]}), as illustrated by the following theorem:

\begin{theorem} \label{sinya}
Let $L\subseteq \sss$ be a rational language. 
The following are equivalent:
\begin{itemize}
\item $\sss s\sss\cap L=\emptyset$ for some $s\in \sss$;
\item $D(L)=0$.
\end{itemize}
\end{theorem} 

 A simple self-contained proof of Theorem \ref{sinya} was given in \cite{[Kog19]}.
 
 \section{Relative densities}
 The purpose of this section is to prove some basic properties about relative densities.
 
 Given two languages $L_1\subseteq L_2\subseteq \sss$, we define the density of $L_1$ relative to $L_2$ as 
 $$D(L_1\mid L_2)=\lim_{n\to \infty}\frac{|\Sigma^n\cap L_1|}{|\Sigma^n\cap L_2|},$$
 when it exists.
 Similar to the case of the densities, we define  $\oD(L_1\mid L_2)=\limsup_{n\to \infty}\frac{|\Sigma^n\cap L_1|}{|\Sigma^n\cap L_2|},$ and $D_{\inf}(L_1\mid L_2)=\liminf_{n\to \infty}\frac{|\Sigma^n\cap L_1|}{|\Sigma^n\cap L_2|}$. 
 
 Analogous to the simple case, we can define the cumulative density relative to a language $L_2$ as $$D^*(L_1\mid L_2)=\lim_{n\to+\infty}\frac{\sum_{i=0}^n|\Sigma^i\cap L_1|}{\sum_{i=0}^n |\Sigma^i\cap L_2|}.$$
 In the cases where $L_2\cap \Sigma^n\neq \emptyset$ for all $n\in \N$ and $$\lim_{n\to+\infty} \frac{\sum_{i=0}^n |\Sigma^i\cap L_2|}{\sum_{i=0}^{n+1} |\Sigma^i\cap L_2|}=b$$ for a real number $b\neq 1$, we can still apply \cite[Theorems 1.22 and 1.23]{[Mur09]} to get that $D(L_1\mid L_2)$ exists if and only if $D^*(L_1\mid L_2)$ exists and, in case they exist, the two notions coincide. Since we will be dealing mostly with the case where $L_2$ is the language of reduced words in a free group, the two notions are interchangeable to our purpose. However, the same is not true for $D_{\sup}(L_1\mid L_2)$ (resp. $D_{\inf}(L_1\mid L_2)$) and $D_{\sup}^*(L_1\mid L_2)$ (resp. $D_{\inf}^*(L_1\mid L_2)$), as it will be clear later.

 \begin{lemma}\label{dod}
 	Let $L_1,L_2, L_3\subseteq \sss$ be languages such that $L_1\cup L_2\subseteq L_3$. Then the following hold:
 	\begin{itemize}
 		\item $\oD(L_1\mid L_3)=0\iff D(L_1\mid L_3)=0$;
 		\item $\oD(L_1\cup L_2\mid L_3)\leq \oD(L_1\mid L_3)+\oD(L_2\mid L_3)$.
 	\end{itemize} 
 \end{lemma}
 \begin{proof}
 	It is clear that $ D(L_1\mid L_3)$ implies $\oD(L_1\mid L_3)=0$. The converse follows from the fact that the sequence $(a_n)_n$, where $a_n=\frac{|\Sigma^n\cap L_1|}{|\Sigma^n\cap L_3|}$ is nonnegative, so if $\limsup_{n\to \infty}a_n=0$, then 
 	$0\leq \liminf_{n\to \infty} a_n\leq \limsup_{n\to \infty}a_n=0$, so $\lim_{n\to \infty}a_n=0$.
 	
 	The second claim follows easily from the fact that $|\Sigma^n\cap (L_1\cup L_2)|\leq |\Sigma^n\cap L_1| + |\Sigma^n\cap L_2|.$
 \end{proof}
 \begin{lemma}
 Let $L_1,L_2\subseteq \sss$ be languages such that $L_1\subseteq L_2$. Then, $D(L_1\mid L_2)=0\iff D(L_2\setminus L_1\mid L_2)=1$.
 \end{lemma}
 \begin{proof}
 Suppose that $D(L_1\mid L_2)=0,$ that is, that
 $$\forall \varepsilon >0 \; \exists N_\varepsilon \in \N: \; \forall n>N_\varepsilon, \; \frac{|\Sigma^n\cap L_1|}{|\Sigma^n\cap L_2|}<\varepsilon.$$
 
 Let $\varepsilon >0$ and take $n>N_\varepsilon$. We have that 
 \begin{align*}
 \frac{|\Sigma^n\cap (L_2\setminus L_1)|}{|\Sigma^n\cap L_2|}=\frac{|\Sigma^n\cap L_2|-|\Sigma^n \cap L_1|}{|\Sigma^n\cap L_2|}=1-\frac{|\Sigma^n\cap L_1|}{|\Sigma^n\cap L_2|}>1-\varepsilon.
 \end{align*}
 Hence,
  $$\forall \varepsilon >0 \; \exists N_\varepsilon \in \N: \; \forall n>N_\varepsilon, \; \frac{|\Sigma^n\cap (L_2\setminus L_1)|}{|\Sigma^n\cap L_2|}>1-\varepsilon.$$
  Since  $\frac{|\Sigma^n\cap (L_2\setminus L_1)|}{|\Sigma^n\cap L_2|}<1$, for all $n\in \N$, it follows that $D(L_2\setminus L_1\mid L_2)=1$.
  
  Now, suppose that $D(L_2\setminus L_1\mid L_2)=1$. Then,
 $$\forall \varepsilon >0 \; \exists N_\varepsilon \in \N: \; \forall n>N_\varepsilon, \; \frac{|\Sigma^n\cap (L_2\setminus L_1)|}{|\Sigma^n\cap L_2|}>1-\varepsilon.$$
 Let $\varepsilon >0$ and take $n>N_\varepsilon$. We have that 
 \begin{align*}
 \frac{|\Sigma^n\cap L_1|}{|\Sigma^n\cap L_2|}=1- \left(\frac{|\Sigma^n\cap L_2|}{|\Sigma^n\cap L_2|} - \frac{|\Sigma^n\cap L_1|}{|\Sigma^n\cap L_2|}\right)
 =1- \frac{|\Sigma^n\cap (L_2\setminus L_1)|}{|\Sigma^n\cap L_2|}>1-(1-\varepsilon)=\varepsilon,
 \end{align*}
 and so $D(L_1\mid L_2)=0.$
 \end{proof}
 
 We now prove that, if a language $L_1$ has density $1$ and another language $L_2$ is \emph{visible}, in the sense that it has positive density, then the intersection $L_1\cap L_2$ has density $1$ in $L_2$, and that that is not the case if $L_2$ has density $0$.
 \begin{proposition}\label{intersect}
Let $L_1,L_2\subseteq \sss$ be languages. If $D(L_1)=1$ and $D(L_2)=a>0$, then $D(L_1\cap L_2\mid L_2)=1$.
 \end{proposition}
\begin{proof}
Suppose that $D(L_1)=1$ and $D(L_2)=a>0$, that is, that
\begin{align}\label{l1}
 \forall \varepsilon >0 \; \exists N_\varepsilon \in \N: \; \forall n>N_\varepsilon, \; \frac{|\Sigma^n\cap L_1|}{|\Sigma^n|}>1-\varepsilon.
\end{align}
and
\begin{align} \label{l2}
 \forall \varepsilon >0 \; \exists N'_\varepsilon \in \N: \; \forall n>N'_\varepsilon, \; a-\varepsilon<\frac{|\Sigma^n\cap L_2|}{|\Sigma^n|}<a+\varepsilon.
\end{align}

We want to prove that 
$$ \forall \varepsilon >0 \; \exists M_\varepsilon \in \N: \; \forall n>M_\varepsilon, \; \frac{|\Sigma^n\cap L_1\cap L_2|}{|\Sigma^n\cap L_2|}>1-\varepsilon.$$

Let $0<\varepsilon<\frac 23$ and take $0<\varepsilon_1<\varepsilon a$ and $0<\varepsilon_2<\frac{\varepsilon a-\varepsilon_1}{2-\varepsilon}$.
Put $M_\varepsilon>\max\{N_{\varepsilon_1},N_{\varepsilon_2}'\}$ and let $n>M_\varepsilon$. Then,
\begin{align*}
|\Sigma^n\cap L_1\cap L_2|&=|\Sigma^n\cap L_2|-|\Sigma^n\cap L_1^c\cap L_2|\\
&\geq |\Sigma^n\cap L_2|- |\Sigma^n\cap L_1^c|\\
&> |\Sigma^n|(a-\varepsilon_2)-|\Sigma^n\cap L_1^c|\\
&>|\Sigma^n|(a-\varepsilon_1-\varepsilon_2).
\end{align*}
Hence, it follows that
\begin{align*}
\frac{|\Sigma^n\cap L_1\cap L_2|}{|\Sigma^n\cap L_2|}&>\frac{|\Sigma^n|(a-\varepsilon_1-\varepsilon_2)}{|\Sigma^n\cap L_2|}\\
&>\frac{a-\varepsilon_1-\varepsilon_2}{a+\varepsilon_2}\\
&=1-\frac{\varepsilon_1+2\varepsilon_2}{a+\varepsilon_2}\\
&=1-\frac{\varepsilon_1+\varepsilon \varepsilon_2+ (2-\varepsilon)\varepsilon_2}{a+\varepsilon_2}\\
&>1-\frac{\varepsilon_1+\varepsilon \varepsilon_2+ \varepsilon a-\varepsilon_1}{a+\varepsilon_2}\\
&=1-\varepsilon.
\end{align*}
\end{proof}
 
 \begin{remark}\label{intersection density zero}
 We might have $D(L_1)=1$ and $D(L_1\cap L_2\mid L_2)=0$. For example, if $\Sigma=\{a,b\}$, $L_1=\sss b\sss$, and  $L_2=a^*$, then $L_1\cap L_2=\emptyset$, so $D(L_1\cap L_2\mid L_2)=0$, but $D(L_1)=1$ by Theorem \ref{sinya}. However, the previous proposition shows that this cannot happen if $L_2$ has positive density.
 \end{remark}

\section{Subshifts of finite type}
 
 Let $\Sigma$ be a finite alphabet. We consider the set $\Sigma^\Z$ of bi-infinite words over $\Sigma$. A \emph{subshift} is a closed subset of $\Sigma^\Z$ invariant under the shift operation $\sigma$, which moves  the sequence one step to the left, that is, if $x=(x_i)_{i\in \Z}$, the shift is defined as
 $(\sigma(x))_i=x_{i+1}$. The \emph{language $\mc L(X)$ of a subshift $X$} is the set of all finite subwords of words in $X$. A word is \emph{forbidden} if it belongs to $\Sigma^*\setminus \mc L(X)$. A way to define a subshift is to fix a set of forbidden words and consider the bi-infinite words over $\Sigma$ that avoid them. A subshift is said to be of \emph{finite type} if it is defined by a finite set of forbidden words. A subshift $X$ is \emph{irreducible} if, for any pair of words $u,v\in \mc L(X)$, there exists $w\in \mc L(X)$ such that $uwv\in \mc L(X)$. The \emph{entropy} of a subshift $X$ is defined as $$h(X)=\lim_{n\to \infty}\frac{\log(|\Sigma^n\cap \mc L(X)|)}{n}.$$
 A subshift is said to be \emph{entropy-minimal} if, for every subshift $Y\subsetneq X$, we have $h(Y)<h(X)$. Irreducible subshifts of finite type are examples of entropy minimal subshifts.
 The entropy of an irreducible subshift of finite type $X$ is $\log(\lambda)$, where $\lambda$ is the leading (or Perron-Frobenius) eigenvalue of the adjacency matrix of the edge shift from which $X$ is conjugate. For more details on symbolic dynamics, we refer the reader to \cite{[LM95]}.
 
 We will now prove that, for two subshifts of finite type where one is contained in the other, the density of the language of the smallest one relative to that of the largest one is $0$. 
 
 \begin{theorem}\label{thm:shifts}
 	Let $X_1, X_2$ be  irreducible subshifts of finite type such that $X_1\subsetneq X_2$. Then $D(\mathcal L(X_1)\mid \mathcal L(X_2))=0$.
 \end{theorem}
 \begin{proof}
 	We have that there are $\lambda_1<\lambda_2$ such that  $\log(\lambda_1)=h(X_1)<h(X_2)=\log(\lambda_2)$. So,  $\lim_{n\to \infty}\frac{\log(|\Sigma^n\cap \mc L(X_1)|)}{n}=\log(\lambda_1)$, which means that, for all $\varepsilon >0$, there exists an integer $N_1$ such that, for all $n> N_1$, $\log(\lambda_1)-\varepsilon <\frac{\log(|\Sigma^n\cap \mc L(X_1)|)}{n}< \log(\lambda_1)+\varepsilon$, that is $\lambda_1^n e^{-n\varepsilon}<|\Sigma^n\cap \mc L(X_1)|<\lambda_1^n e^{n\varepsilon}$. Similarly, for all $\varepsilon >0$, there is some $N_2$ such that, for all $n>N_2$, 
 	$\lambda_2^n e^{-n\varepsilon}<|\Sigma^n\cap \mc L(X_2)|<\lambda_2^n e^{n\varepsilon}$. So, we get that, for all $\varepsilon>0$, taking $N=\max\{N_1,N_2\}$ and $n>N$, 
 	$$0\leq \frac{|\Sigma^n\cap \mc L(X_1)|}{|\Sigma^n\cap \mc L(X_2)|}<\frac{\lambda_1^n e^{n\varepsilon}}{\lambda_2^n e^{-n\varepsilon}}=\left(\frac{\lambda_1e^{2\varepsilon}}{\lambda_2}\right)^n.$$
 	In particular, for $\varepsilon <\frac{\log(\lambda_2)-\log(\lambda_1)}{2}$, we get that, $\frac{\lambda_1e^{2\varepsilon}}{\lambda_2}<1$, so $\left(\frac{\lambda_1e^{2\varepsilon}}{\lambda_2}\right)^n\to 0$. Hence, $D(\mathcal L(X_1)\mid \mathcal L(X_2))=\lim_{n\to \infty}\frac{|\Sigma^n\cap \mc L(X_1)|}{|\Sigma^n\cap \mc L(X_2)|}=0$
 \end{proof}
 
 \section{Free groups}
 We will now focus on subsets of free groups. We will prove a version of the infinite monkey theorem for reduced words using the previous result on subshifts and prove the converse for rational languages of reduced words. We use this result to show that the automorphic orbit of any element has natural density $0$, generalizing Burillo-Ventura's result on primitive elements \cite{[BV02]} and  describe rational subsets of positive density in a nonabelian free group. We then focus on the subgroup case where we describe those of positive density and those for which there is convergence.
 
  First, we show an analogue of the infinite monkey theorem for reduced words using Theorem \ref{thm:shifts}. In words, we show that a monkey that only writes reduced words will eventually write any reduced word with probability 1. We remark that the language of reduced words has density $0$ by the infinite monkey theorem, as no reduced word contains a subword of the form $aa\inv$, so the result does not follow immediately from Proposition \ref{intersect} (see Remark \ref{intersection density zero}).

 We denote by $\Red(X)$ the language of freely reduced words over $\Sigma=X\cup X^{-1}$. It is easy to see that $|\Sigma^n\cap \Red(X)|=2|X|(2|X|-1)^{n-1}$ as there are $2|X|$ options for the first letter of a reduced word and in successive letters, there is only one letter that makes the word not reduced, so we always have $2|X|-1$ options.

 \begin{corollary}\label{infinitemonkey reduced}
 	Let $k\geq 2$, $X=\{x_1,\ldots, x_k\}$ and $L\subseteq \Red(X)$.  If $\sss s\sss\cap L=\emptyset$ for some $s\in \Red(X)$, then $D(L\mid \Red(X))=0$.
 \end{corollary}
 \begin{proof}
 	Let the set $X_2$ of all bi-infinite reduced words over $\Sigma$, which is clearly an irreducible subshift of finite type, and $X_1$ be the set of all bi-infinite reduced words over $\Sigma$ that avoid $s$. Then $X_1$ is also a subshift of finite type. We will now see that it is irreducible. To prove $X_1$ is irreducible, let $u, v \in \mathcal{L}(X_1)$,  $m=|s|$ and let $s_{\text{first}}$ and $s_{\text{last}}$ denote the first and last letters of $s$, respectively. Since $|\Sigma| \geq 4$, we can choose $z$ such that $uz$ is reduced and $z \notin \{s_{\text{first}}, s_{\text{last}}\}$. Similarly, we choose $y$ such that $zyv$ is reduced and $y \neq s_{\text{first}}$. The word $u z^m y^m v$ is reduced and avoids $s$. Indeed, an occurrence of $s$ cannot start inside $u$ as it cannot end inside $u$, since $s$ is not a subword of $u$ and cannot end in $z$ as $z\neq s_{\textbf{last}}$; it cannot start in $z$ or $y$ as $z\neq s_\text{first}$ and  $y\neq s_\text{first}$; it cannot start inside $v$ as $s$ is not a subword of $v$. Hence, by  Theorem \ref{thm:shifts}, $D(\mathcal L(X_1)\mid \mathcal L(X_2))=0$.
 	
 	But $\mc L(X_2)=\Red(X)$ as, clearly, every reduced word can be extended to a bi-infinite reduced word. We will now show that $L\subseteq \mc L(X_1)$, which yields that $$D(L\mid \Red(X))=D( L\mid \mathcal L(X_2))\leq D(\mathcal L(X_1)\mid \mathcal L(X_2))=0.$$
 	
 	To do so, it suffices to see that any reduced word avoiding $s$ can be extended to a word in $X_1$, that is, a bi-infinite reduced word avoiding $s$.
 	Let $u$ be such a word, $u_1$ be the longest  prefix of $u$ which is a suffix of $s$, $u_3$ be the longest suffix of $u$ which is a prefix of $s$ and $u_2$ be such that $u=u_1u_2u_3$. Since $X\geq 2$, there are always letters $a,b\in \Sigma$ such that $au_1$ and $u_3b$ are reduced, $au_1$ is not a suffix of $s$ and $u_3b$ is not a prefix of $s$. Also, there are letters $x,y\in\Sigma$ different from the last and first letters of s, respectively, such that $xau_1u_2u_3by$ is reduced. We claim that the bi-infinite word $\cdots xxxau_1u_2u_3byyy\cdots$ is reduced and does not contain $s$, that is belongs to $X_1$. Reduction comes from the choice of $x,a,b$ and $y$. The factor $u=u_1u_2u_3$ does not contain $s$ as $u\in L$. Since $y$ is not the first letter of $s$, then a potential occurrence of $s$ would have to start before $b$. Since $u_3b$  is not a prefix of $s$ and $u_3$ is the longest prefix of $s$ that is a suffix of $u$, then a potential occurrence of $s$ start before $u$. But, by a similar reasoning, since $x$ is not the last letter of $s$, such an occurrence would have to end inside $u$, which cannot happen since $u_1$ is the longest prefix of $u$ that is a suffix of $s$ and $au_1$ is not a suffix of $s$.
 \end{proof}

 We can define $D_{\sup}^*(L\mid \Red(X))$ and $D_{\inf}^*(L\mid \Red(X))$ in the obvious way.
 We remark that although $D(L \mid \Red(X))=D^*(L\mid \Red(X))$, it might be the case that 
 $D_{\sup}^*(L\mid \Red(X))$ (resp. $D_{\inf}^*(L\mid \Red(X))$) does not coincide with $D_{\sup}(L\mid \Red(X))$ (resp. $D_{\inf}(L\mid \Red(X))$). However, positivity cannot change, as shown by the next proposition.
 
 \begin{proposition}\label{null cumulative}
 	Let $L\subseteq \Red(X)$ be a language of reduced words. Then $D_{\sup}^*(L\mid \Red(X))=0$ if and only if $D_{\sup}(L\mid \Red(X))=0$.
 \end{proposition}
 \begin{proof}
 	Suppose that $D_{\sup}(L\mid \Red(X))=0$, that is, that $\limsup_{n\to \infty} \frac{|\Sigma^n\cap L|}{|\Sigma^n\cap \Red(X)|}=0.$
 	So, we have that 
 	$$\forall \varepsilon >0 \;\exists N\in \N \; : \forall n>N, \; \frac{|\Sigma^n\cap L|}{|\Sigma^n\cap \Red(X)|}<\varepsilon.$$
 	Now, let $\varepsilon >0$ and $M$ be such that $ \frac{|\Sigma^n\cap L|}{|\Sigma^n\cap \Red(X)|}<\frac{\varepsilon}{2},$ for all $n>M$. Also, let $M'$ be such that $ \frac{\sum_{i=0}^M|\Sigma^i\cap \Red(X)|}{\sum_{i=0}^n |\Sigma^i\cap \Red(X)|} <\frac \varepsilon 2$ and put $N=\max\{M,M'\}$ and let $n>N$. We will prove that $$\frac{\sum_{i=0}^n|\Sigma^i\cap L|}{\sum_{i=0}^n |\Sigma^i\cap \Red(X)|}<\varepsilon,$$
 	thus showing that 
 		$$\forall \varepsilon >0 \;\exists N\in \N \; : \forall n>N, \; \frac{\sum_{i=0}^n|\Sigma^i\cap L|}{\sum_{i=0}^n |\Sigma^i\cap \Red(X)|}<\varepsilon,$$
 		that is, that $D_{\sup}^*(L\mid \Red(X))=\limsup_{n\to \infty} \frac{\sum_{i=0}^n|\Sigma^i\cap L|}{\sum_{i=0}^n |\Sigma^i\cap \Red(X)|}=0$.
 We have that 
 \begin{align*}
 	\frac{\sum_{i=0}^n|\Sigma^i\cap L|}{\sum_{i=0}^n |\Sigma^i\cap \Red(X)|} &=	\frac{\sum_{i=0}^M|\Sigma^i\cap L|+\sum_{i=M+1}^n|\Sigma^i\cap L|}{\sum_{i=0}^n |\Sigma^i\cap \Red(X)|}\\
 	&\leq \frac{\sum_{i=0}^M|\Sigma^i\cap \Red(X)|+\sum_{i=M+1}^n|\Sigma^i\cap L|}{\sum_{i=0}^n |\Sigma^i\cap \Red(X)|}\\
 	&\leq \frac{\sum_{i=0}^M|\Sigma^i\cap \Red(X)|+ \sum_{i=M+1}^n \frac \varepsilon 2|\Sigma^i\cap \Red(X)|}{\sum_{i=0}^n |\Sigma^i\cap \Red(X)|}\\
 	&\leq  \frac{\sum_{i=0}^M|\Sigma^i\cap \Red(X)|}{\sum_{i=0}^n |\Sigma^i\cap \Red(X)|} +\frac \varepsilon 2\frac{\sum_{i=M+1}^n |\Sigma^i\cap \Red(X)|}{\sum_{i=0}^n |\Sigma^i\cap \Red(X)|}\\
 	&\leq \frac \varepsilon 2 + \frac \varepsilon 2\frac{\sum_{i=0}^n |\Sigma^i\cap \Red(X)|}{\sum_{i=0}^n |\Sigma^i\cap \Red(X)|}\\
 	&= \varepsilon.
 \end{align*}
 	 	To prove the converse, suppose that 
 	 		$$\forall \varepsilon >0 \;\exists N\in \N \; : \forall n>N, \; \frac{\sum_{i=0}^n|\Sigma^i\cap L|}{\sum_{i=0}^n |\Sigma^i\cap \Red(X)|}<\varepsilon.$$
 	Let $\varepsilon>0$ and $N\in \N$ be such that  $\frac{\sum_{i=0}^n|\Sigma^i\cap L|}{\sum_{i=0}^n |\Sigma^i\cap \Red(X)|}<\frac{2}{3}\varepsilon$ for all $n>N$. If there was $n>N$ such that $\frac{|\Sigma^n\cap L|}{|\Sigma^n\cap \Red(X)|}\geq\varepsilon$, then
 	\begin{align*}
 		\frac{\sum_{i=0}^n|\Sigma^i\cap L|}{\sum_{i=0}^n |\Sigma^i\cap \Red(X)|}&\geq 	\frac{|\Sigma^n\cap L|}{\sum_{i=0}^n |\Sigma^i\cap \Red(X)|}\\
 		&\geq \frac{\varepsilon |\Sigma^n\cap \Red(X)|}{\sum_{i=0}^n |\Sigma^i\cap \Red(X)|}\\
 		&= \varepsilon \frac{2|X|(2|X|-1)^{n-1}}{{\frac{|X|(2|X|-1)^{n}-1}{|X|-1}}}\\
 		&>\varepsilon  \frac{2|X|(2|X|-1)^{n-1}}{{\frac{|X|(2|X|-1)^{n}}{|X|-1}}}\\
 			&=\varepsilon \frac{2|X|-2}{2|X|-1}\\
 			&\geq \frac2 3 \varepsilon,
 	\end{align*}
 	which contradicts the choice of $N$.
 \end{proof}
 
  In \cite{[BV02]}, the authors define the \emph{natural density} of a subset $S\subseteq G$ of a finitely generated group $G$ with finite generating set $X$ as $$\delta_X(S)=\limsup_{n\to \infty}\frac{|S\cap B_X(n)|}{|B_X(n)|},$$ where $B_X(n)$ represents the ball of radius $n$ (centered at the identity) with respect to $X$. They proceed to prove that the set of primitive words of a free group has density 0.

We now prove useful lemma for us to reach the goal of this section, that is, proving that the automorphic orbit of an element of a nonabelian free group has density $0$.

\begin{lemma}\label{lem: cycredcore}
	Let $X=\{x_1,\ldots, x_k\}$ be a basis for a nonabelian free group $F_k$, $S\subseteq F_k$ be a subset of $F_k$ such that $\delta(S)=0$ and let 
	$$S'=\{u\inv s u\mid u\in F_k,\, s\in S,\, \text{ and $u\inv s u$ is reduced}\}$$ 
	be the set of all words whose cyclically reduced core lies in $S$. Then $\delta(S')=0$.
\end{lemma}
\begin{proof} 	Let $\overline{S}$ and $\overline{S'}$  be the languages of reduced words representing elements in $S$ and $S'$, respectively.
	Notice that, since the length of a reduced word has the same parity as the length of its cyclically reduced core and that, setting the word on the left of the cyclically reduced core forces the choice of letters appearing to the right of the cyclically reduced core, we have,
	for $n\in \N$, 
	\begin{align*}
		|\Sigma^{2n+1}\cap \overline{S'}|=\sum_{i=0}^n |\Sigma^{2i+1}\cap \overline{S}|(2k-1)^{n-i} \quad \text{ and } \quad |\Sigma^{2n}\cap \overline{S'}|=\sum_{i=1}^n |\Sigma^{2i}\cap \overline{S}|(2k-1)^{n-i}. 
	\end{align*}

	To show that $\delta(S')=D_{\sup}^*(\overline{S'}\mid\Red(X))=0$, we will show that $D_{\sup}(\overline{S'}\mid\Red(X))=0$, which suffices by Proposition \ref{null cumulative}. We will see that both the even and odd subsequences of $\frac{|\Sigma^n\cap \overline{S'}|}{|\Sigma^n\cap \Red(X)|}$ converge to $0$. We will start with the odd case, that is, showing that $$	\limsup_{n\to +\infty}\frac{|\Sigma^{2n+1}\cap \overline{S'}|}{|\Sigma^{2n+1}\cap \Red(X)|}=0.$$
	Let $\varepsilon>0$. Since $\delta(S)=0$, then $\oD^*(\overline{S}\mid \Red(X))=0$ and so $\oD(\overline{S}\mid \Red(X))=0$ by Proposition \ref{null cumulative}. Hence, there is some $M\in \N$ such that, for all $n>M$, $\frac{|\Sigma^{2n+1}\cap \overline{S}|}{|\Sigma^{2n+1}\cap \Red(X)|}<\varepsilon \frac{2k-2}{2(2k-1)}.$ Now, for all $n>M$ and $n$ large enough so that $\sum_{i=0}^M\frac{ |\Sigma^{2i+1}\cap \overline{S}|}{2k(2k-1)^{n+i}}<\frac \varepsilon 2$ , we have that 
	\begin{align*}
\frac{|\Sigma^{2n+1}\cap \overline{S'}|}{|\Sigma^{2n+1}\cap \Red(X)|}
=&\frac{\sum_{i=0}^n |\Sigma^{2i+1}\cap \overline{S}|(2k-1)^{n-i}}{2k(2k-1)^{2n}}\\
=&\frac{\sum_{i=0}^M |\Sigma^{2i+1}\cap \overline{S}|(2k-1)^{n-i}+ \sum_{i=M+1}^n |\Sigma^{2i+1}\cap \overline{S}|(2k-1)^{n-i}}{2k(2k-1)^{2n}}\\
=&\sum_{i=0}^M\frac{ |\Sigma^{2i+1}\cap \overline{S}|}{2k(2k-1)^{n+i}}+ \sum_{i=M+1}^n\frac{ |\Sigma^{2i+1}\cap \overline{S}|}{2k(2k-1)^{2i}}\frac{1}{(2k-1)^{n-i}}\\
<&\frac \varepsilon 2+ \varepsilon \frac{2k-2}{2(2k-1)}\sum_{i=M+1}^n\frac{1}{(2k-1)^{n-i}}\\
=&\frac \varepsilon 2+ \varepsilon \frac{2k-2}{2(2k-1)}\sum_{i=0}^{n-M-1}\frac{1}{(2k-1)^{i}}\\
<&\frac \varepsilon 2+ \varepsilon \frac{2k-2}{2(2k-1)}\sum_{i=0}^\infty\frac{1}{(2k-1)^{i}}\\
=&\varepsilon
	\end{align*}
	
	Similarly, for the even case, we aim to show that 
	$$ \limsup_{n\to +\infty}\frac{|\Sigma^{2n}\cap \overline{S'}|}{|\Sigma^{2n}\cap \Red(X)|}=0. $$
	Given $\varepsilon>0$, we use the fact that $\oD(\overline{S}\mid \Red(X))=0$ to find $M' \in \N$ such that for all $n > M'$, we have $\frac{|\Sigma^{2n}\cap \overline{S}|}{|\Sigma^{2n}\cap \Red(X)|} < \varepsilon \frac{2k-2}{2(2k-1)}$. For $n$ sufficiently large such that $\sum_{i=1}^{M'}\frac{|\Sigma^{2i}\cap \overline{S}|}{2k(2k-1)^{n+i-1}}<\frac{\varepsilon}{2}$, we have:
	\begin{align*}
		\frac{|\Sigma^{2n}\cap \overline{S'}|}{|\Sigma^{2n}\cap \Red(X)|}
		&=\frac{\sum_{i=1}^n |\Sigma^{2i}\cap \overline{S}|(2k-1)^{n-i}}{2k(2k-1)^{2n-1}}\\
		&=\sum_{i=1}^{M'}\frac{|\Sigma^{2i}\cap \overline{S}|}{2k(2k-1)^{n+i-1}} + \sum_{i=M'+1}^n \frac{|\Sigma^{2i}\cap \overline{S}|}{2k(2k-1)^{2i-1}}\frac{1}{(2k-1)^{n-i}}\\
		&< \frac{\varepsilon}{2} + \varepsilon \frac{2k-2}{2(2k-1)} \sum_{j=0}^{n-M'-1} \frac{1}{(2k-1)^{j}}\\
		&< \frac{\varepsilon}{2} + \varepsilon \frac{2k-2}{2(2k-1)} \frac{2k-1}{2k-2} = \varepsilon.
	\end{align*}
	Since both the odd and even subsequences converge to $0$, it follows that $D_{\sup}(\overline{S'}\mid\Red(X))=0$. By Proposition \ref{null cumulative}, we conclude $\delta(S')=0$.
	\end{proof}

	A word is said to be \emph{primitivity-blocking} if it is not a subword of any cyclically reduced primitive word. It is a consequence of Whitehead's work \cite{[Whi36]} that nonabelian free groups have primitivity-blocking words (see also \cite{[KOOS26]}).  In fact, in \cite{[KOOS26]}, the authors show that for any element $w\in F_k$, there is some  element $g\in F_k$ such that no cyclically reduced image of $w$ under an automorphism of $F_k$ contains $g$ as a subword.

	Hence, Corollary \ref{infinitemonkey reduced} together with Proposition \ref{null cumulative} and the previous lemma yield a language-theoretic proof of
	that automorphic orbits of elements in nonabelian free groups have density zero, which generalizes Burillo-Ventura's result on primitive words.

 \begin{corollary}\label{cor: shpilrain}
 	Let $g$ be an element of a nonabelian free group $F_k$ and $S'$ be the set of   images of $g$ under an automorphism of $F_k$.  Then $\delta(S')=0$.
 \end{corollary}
 \begin{proof}
 Let, $S$ be the set of cyclically reduced automorphic images of $g$, $L$ be the language of reduced words representing  elements in $S$ and $s$  be an orbit-blocking word for $g$. We have that $\sss s \sss \cap L=\emptyset$, so $\oD(L\mid \Red(X))=0$ by Corollary \ref{infinitemonkey reduced}, which means that $D_{\sup}^*(L\mid \Red(X))=0$ by Proposition \ref{null cumulative}  and  $\delta(S')=D_{\sup}^*(L\mid \Red(X))=0$. Now the result follows from Lemma \ref{lem: cycredcore} noting that $$S'=\{u\inv s u\mid u\in F_k,\, s\in S,\, \text{ and $u\inv s u$ is reduced}\}$$. Indeed, any reduced word of the form $uwu^{-1}$ with $w\in S$ lies in $S'$ as, by definition of $S$, there is an automorphism $\phi$ such that $g\phi=w$, so $uwu^{-1}=g\phi\lambda_u\in S'$, where $\lambda_u$ denotes the inner automorphism given by $u$. Conversely, if $z\in S'$, then there is some automorphism $\phi$ such that $z=g\phi$. Also, $z=w^{-1}uw$ for some cyclically reduced word $u$ and $u=z\lambda_{w\inv}=g\phi\lambda_{w\inv}\in S$.

 \end{proof}
 
 \begin{corollary}\label{cor: burillo-ventura}
 	Let $S$ be the set of primitive elements of a nonabelian free group $F_n$. Then $\delta(S)=0$. 
 \end{corollary}

 \subsection{Rational subsets}
 
   Let $F_n$ be a free group with basis $X=\{x_1,\ldots, x_n\}$, $\Sigma=X\cup X^{-1}$ be a set of monoid generators, and $\pi:\Sigma^*\to F_n$ be the canonical surjective homomorphism. 
   
  Following Koga's approach, we can now prove an analogue of Theorem \ref{sinya} in the context of relative densities with respect to reduced words in the free group.

Given $a\in \Sigma$ and $L\subseteq \sss$, we put $a^{-1}L=\{y\in \sss\mid ay\in L\}$ and $La^{-1}=\{y\in \sss\mid ya\in L\}$. We start by presenting a lemma from \cite{[Kog19]} and proving an analogue of \cite[Lemma 2.2]{[Kog19]} for reduced words.

\begin{lemma}\cite[Theorem 2.3]{[Kog19]}
	\label{lemma koga}
	Let $L\subseteq \sss$ be a rational language. Then there exists $p\geq 1$ such that the following two conditions are satisfied:
	\begin{align*}
		&\forall x,y\in\sss \; [xy\in L \implies \exists x'\in \Sigma^{<p}\; [x'y\in L]]\\
		&\forall x,y\in\sss \; [yx\in L \implies \exists x'\in \Sigma^{<p}\; [yx'\in L]].
	\end{align*}
\end{lemma}

\begin{lemma}\label{lema inversos}
	For any $x\in \Red(X)$ and $L\subseteq \Red(X)$, the following hold:
	\begin{itemize}
		\item  $\oD(x^{-1}L\mid \Red(X))\leq \oD(L\mid \Red(X))(2|X|-1)^{|x|}$;
		\item  $D_{\inf}(x^{-1}L\mid \Red(X))\leq D_{\inf}(L\mid \Red(X))(2|X|-1)^{|x|}$;
		\item $\oD(Lx^{-1}\mid \Red(X))\leq \oD(L\mid \Red(X))(2|X|-1)^{|x|}$; 
		\item  $D_{\inf}(Lx^{-1}\mid \Red(X))\leq D_{\inf}(L\mid \Red(X))(2|X|-1)^{|x|}$).
	\end{itemize}
\end{lemma}
\begin{proof}
We start by proving that 
\begin{align}
	\label{eq1}
\oD(x(x^{-1}L)\mid \Red(X))=\oD(x^{-1}L\mid \Red(X))\cdot
(2|X|-1)^{-|x|}.
\end{align} 
 Notice that $x(x^{-1}L)\subseteq L\subseteq \Red(X)$ and that $x^{-1}L\subseteq \Red(X)$ because $\Red(X)$ is factorial.
Let $n>|x|$. Then $x(x^{-1}L)\cap \Sigma^n=x\left((x^{-1}L)\cap \Sigma^{n-|x|}\right)$, so
\begin{align*}
	\frac{|(x(x^{-1}L)\cap \Sigma^n|}{|\Sigma^n\cap \Red(X)|}
	=&\frac{|x\left((x^{-1}L)\cap \Sigma^{n-|x|}\right)|}{2|X|(2|X|-1)^{n-1}}\\
	=&\frac{|x^{-1}L\cap \Sigma^{n-|x|}|}{2|X|(2|X|-1)^{n-|x|-1}}\cdot \frac{1}{(2|X|-1)^{|x|}}\\
	=&\frac{|x^{-1}L\cap \Sigma^{n-|x|}|}{|\Sigma^{n-|x|}\cap \Red(X)|}\cdot \frac{1}{(2|X|-1)^{|x|}}.
\end{align*}
	Taking $\limsup_{n\to \infty}$ in both sides, we get (\ref{eq1}).
	
 Now, since $x(x^{-1}L)\subseteq L$, we have that  $\oD(L\mid \Red(X))\geq \oD(x(x^{-1}L)\mid \Red(X))=\oD(x^{-1}L\mid \Red(X))\cdot
 (2|X|-1)^{-|x|}$, and so, 
 $$\oD(x^{-1}L\mid \Red(X))\leq (2|X|-1)^{|x|} \oD(L\mid \Red(X)).$$
 The cases for $Lx^{-1}$ and for $D_{\inf}$ are analogous.
\end{proof}

We can now prove Theorem \ref{sinya} for reduced words.

 \begin{theorem}\label{sinya analogue}
 	Let $L\subseteq \Red(X)$ be a rational language and let $N=\{w\in \Red(X)\mid \sss w \sss \cap L \neq\emptyset\}$. Then,  the following are equivalent:
 	\begin{enumerate}
 		\item $D_{\sup}^*(L\mid\Red(X))=0;$
 		\item $\oD(L\mid\Red(X))=0;$
 		\item  $D(L\mid \Red(X))=0;$
 		\item  $D(N\mid \Red(X))=0;$
 		\item  $\exists w\in \Red(X) : \sss w \sss \cap L=\emptyset.$
 	\end{enumerate}
 \end{theorem}
 
 \begin{proof}
 	The equivalence between 1 and 2 follows from Proposition \ref{null cumulative} and the one between 2 and 3 from  Lemma \ref*{dod},
 	the implication $5\implies 3$ is given by Corollary \ref{infinitemonkey reduced} and it is obvious that $4\implies 5$. We now prove that $3\implies 4$, that is, we show that if $L\subseteq \Red(X)$ is a rational language such that  $D(L\mid \Red(X))=0$, and    $N=\{w\in \Red(X)\mid \sss w \sss \cap L \neq\emptyset\}$ then $D(N\mid \Red(X))=0.$

 	 Let $L$ be such a language. By Lemma \ref{dod}, we have that $\oD(L\mid \Red(X))=0$. If $N=\emptyset$, we are done. If not, let $s\in N$. Then, there are $a,b\in \sss$ such that $asb\in L$, which, by Lemma \ref{lemma koga} (with $x=a$ and $y=sb$), means that there is some $a'\in \Sigmap$ such that $a'sb\in L$. Now applying Lemma \ref{lemma koga} again (now with $y=a's$ and $x=b$), we have that there must be some $b'\in \Sigmap$ such that $a'sb'\in L$. 
 	
 	Hence, for every $s\in N$, there must be some $a,b\in \Sigmap$ such that $asb\in L$.
 	Equivalently, for all $s\in N$, there are $a,b\in \Sigmap$ such that $s\in a^{-1}Lb^{-1}$. Hence,
 	$$N\subseteq \bigcup_{a,b\in \Sigmap} a^{-1}Lb^{-1}\subseteq \Red(X),$$ so, by Lemma \ref{lema inversos},
 	
 \begin{align*}	
 	\oD(N\mid \Red(X))&\leq \oD\left( \bigcup_{a,b\in \Sigmap} a^{-1}Lb^{-1}\mid \Red(X)\right)\\
 	&\leq \sum_{a,b\in\Sigmap} \oD(a^{-1}Lb^{-1}\mid\Red(X))\\
 	&\leq   \sum_{a,b\in\Sigmap}(2|X|-1)^{|a|+|b|} \oD(L\mid \Red(X))\\
 	&=0.
 	\end{align*}
 	
 	Since $\oD(N\mid \Red(X))\geq 0$, we deduce that $\oD(N\mid\Red(X))=0$, which, by Lemma \ref{dod}, means that $D(N\mid\Red(X))=0$.
 \end{proof}

 By Benois's Theorem, rational subsets of groups are rational languages of reduced words. So, in our terminology, if $L$ is the rational language of reduced words representing the elements in a given rational subset $S\in \Rat(F_X)$ of a nonabelian free group, then $$\delta(S)=D_{\sup}^*(L\mid \Red(X)).$$ In particular, by Theorem \ref{sinya analogue}, $\delta(S)=0$ if and only if $\oDc(L\mid \Red(X))=0$   if and only if $\oD(L\mid \Red(X))=0$. Using the previous result, we can describe rational subsets of the free group with positive density.

 Given a word $w$ over the generators, let $\overline w$ be the reduced word representing the same element as $w$, that is, such that $w\pi=\overline w\pi$.
 
 \begin{theorem}\label{2cosets}
 Let $S\subseteq F_X$ be a rational subset of a nonabelian free group.  Then $\delta(S)>0$ if and only if $F_X$ is there is some $k\in\N$ and finitely many elements $a_i,b_i\in F_X$, with $i\in [k]$ such that 
 $F_X=\bigcup_{i,j\in [k]} a_iSb_j$.
 \end{theorem}
 \begin{proof}
	Suppose that $\delta(S)>0$. By Benois's Theorem, the language $L$ of reduced words representing elements in $S$ is rational. Notice that $\delta(S)=\oDc(L\mid \Red(X))$. Let $\mathcal{A}=\{Q,q_0, F, \delta\}$ be the minimal automaton recognizing $L$. For each state $q\in Q$, take $i_q,f_q\in \Red(X)$ such that 
 	$$q_{0}\xrightarrow{i_q}q\xrightarrow{f_{q}}q_{f}$$
 	is a walk in $\mathcal{A}$, for some $q_f\in F$.
 		
 	 We claim that $F_X=\bigcup_{(q,q')\in Q\times Q} i_q^{-1} S f_q^{-1}$. Let $w\in \Red(X)$. Since $\oDc(L\mid \Red(X))>0$, by Theorem \ref{sinya analogue}, there are some $x,y\in\Red(X)$ such that $xwy\in L$. Let $q,q'\in Q$ be such that  
 	  	$$q_{0}\xrightarrow{x}q\xrightarrow{w}q'\xrightarrow{y}q_{f}$$
 	  	is a path in $\mathcal{A}$, for some $q_f\in F$. By definition of $i_q$ and $f_{q'}$, the word $i_q w f_{q'}$ also belongs to $L$. Hence, $w\in i_q\inv S f_{q'}\inv$. 	
 	  	
 	  	Conversely, let $S$ be a rational subset of $F_X$ such that $F_X=\bigcup_{i,j\in [k]} a_iSb_j$ for some $k\in \N$ and $a_i,b_i\in F_X$ for $i\in [k]$ and let $L$ be the language of reduced words representing elements in $S$, which is rational by Benois's Theorem. Suppose that $\delta(S)=0$, that is, that $$\oD(L\mid\Red(X))=0.$$ We will show that, in this case,  $\delta(a_iSb_j)=0$, for all $i,j\in K$, which   by Lemma \ref{dod} would imply that $\delta(S)\leq \sum_{i,j\in [k]} \delta(a_iSb_j)=0$, a contradiction.
 	  	
		For $i, j\in [k]$ and $s\in L$, we have that $s=\overline{a_i\inv(a_isb_j)b_j\inv}$, and so, that 
		$$|s|\leq |a_i|+|b_j|+ |\overline{a_isb_j}|.$$
		Hence, 
 	  	$$|\Sigma^n\cap a_iSb_j|\leq |\Sigma^{n+|a_i|+|b_j|}\cap S|$$
 	  	and 
 	  	\begin{align*}
 	    \limsup_{n\to \infty} \frac{|\Sigma^n\cap a_iSb_j|}{|\Sigma^n\cap \Red(X)|}
 	  	&\leq \limsup_{n\to \infty} \frac{|\Sigma^{n+|a_i|+|b_j|}\cap S|}{|\Sigma^n\cap \Red(X)|}\\
 	  	&=\limsup_{n\to \infty} \left(\frac{|\Sigma^{n+|a_i|+|b_j|}\cap S|}{|\Sigma^{n+|a_i|+|b_j|}\cap \Red(X)|}\right)\left( \frac{|\Sigma^{n+|a_i|+|b_j|}\cap \Red(X)|}{|\Sigma^n\cap \Red(X)|}\right)\\
 	  	&= (2k-1)^{|a_i|+|b_j|}\delta(S)\\
 	  	&=0,
 	  	\end{align*}
 	  	so $\delta(a_iSb_j)=0$.
 \end{proof}

 \subsection{Finitely generated subgroups}
 It is proved in \cite{[BV02]} that, under some  assumptions on the growth of a group $G$, subgroups $H$ of finite index have density $\frac{1}{[G:H]}$ and that, in that case, the limit always exists. However, in the free group, taking for example the subgroup $H$ of words of even length,  we get that $H$ is a subgroup of index two, and denoting by $L$ the language of reduced words representing elements of $H$, we get that but $D^*(L\mid \Red(\{a,b\}))$ does not exist as  $\oD^*(L\mid \Red(\{a,b\}))=\frac 3 4$ and $D_{\inf}^*(L\mid \Red(\{a,b\}))=\frac{1}{4}$.
 
\begin{remark}
	This subgroup gives us an example where $\oD(L\mid \Red\{a,b\})\neq \oD^*(L\mid \Red\{a,b\})$ and $D_{\inf}(L\mid \Red\{a,b\})\neq D_{\inf}^*(L\mid \Red\{a,b\})$, as it is easy to see that $D_{\inf}(L\mid \Red\{a,b\})=0$ and $D_{\sup}(L\mid \Red\{a,b\})=1$.
\end{remark} 
 
  We now give a language-theoretic proof that the subgroups of positive density in a nonabelian free group are precisely the finite index ones.

 \begin{corollary}\label{cor: fi}
 	 Let $H\leq_{f.g.} F_X$ be a finitely generated subgroup of a nonabelian free group.  Then $\delta(H)>0$ if and only if $H$ has finite index.
\end{corollary}
\begin{proof}
	By Theorem \ref{2cosets}, there is a natural number $k\in \N$ and  finitely many elements $a_i,b_i\in F_X$, with $i\in [k]$ such that 
	\begin{align*}
	F_X=&\bigcup_{i,j\in [k]} a_iHb_j\\
	=&\bigcup_{i,j\in [k]} a_ib_j(b_j\inv Hb_j)
\end{align*}
By a Lemma of Neumann's \cite[Lemma 4.1]{[Neu54]}, we have that if $G$ is a union of finitely many cosets o subgroups, then at least one of the subgroup involved must be of finite index. Since all subgroups $b_j\inv H b_j$ are isomorphic (and isomorphic to $H$), we deduce that $H$ has finite index.

Conversely, if $H$ is a finite index subgroup, then its Stallings automaton is saturated, and so the language of reduced words representing elements of $H$ has no forbidden subword.
\end{proof}

\begin{remark}
	We remark that, as rational subsets generalize finitely generated subgroups,  recognizable subsets generalize finite index subgroups. However, there are rational subsets of $F_2$ with positive density that are not recognizable. For instance, the set $S$ of all reduced words starting with $a$ clearly has positive density as it has no forbidden reduced subword. It is easy to see that $S$ is not recognizable, so rational subsets of positive density properly generalize the notion of recognizable subsets.
\end{remark}
 
Notice that, in the above example, while we don't have that $D(L\mid\Red(\{a,b\}))=\frac{1}{[F_2:H]}$, it is true that $$\frac{D_{\inf}(L\mid\Red(\{a,b\}))+D_{\sup}(L\mid\Red(\{a,b\}))}{2}=\frac{1}{[F_2:H]}.$$ We will show that this always holds in a nonabelian free group and that we have convergence if and only if the finite index subgroup has some word of odd length, that is, if is not contained in $H$.

\begin{proposition}\label{even bipartite}
	Let $H\leq F_n$ be a subgroup of a free group and $\mc G$ be its Stallings graph. The following are equivalent:
	\begin{enumerate}
		\item $H\leq F_n^{even}$;
		\item all closed loops in $\mc G$ have even length;
		\item $\mc G$ is bipartite.
	\end{enumerate} 
\end{proposition}
\begin{proof}
	Since $\mc G$ is the Stallings graph corresponding to $H$, then, if $H\leq F_n^{even}$, then every closed loop at the basepoint must have even length and so every closed loops must have even length: if there was a loop of odd length  labelled by $w$ somewhere in $\mc G$, since $\mc G$ is connected, we would have a closed path at the basepoint  labelled by $uwu^{-1}$, and so, of odd length. So, we have shown that $1	\Rightarrow 2$.
	
	Let $V$ be the set of vertices in $\mc G$. If all closed loops in $\mc G$ have even length, define $E=\{v\in V\mid d(1,v) \text{ is even}\}$ and $O=V\setminus E$. Suppose that there is an edge $p \xrightarrow{e} q$ for $p,q\in E$. Take geodesic paths $u,v$ from $1$ to $p$ and from $q$ to $1$, respectively (notice that $u$ and $v$ have even length). Then, 
	$1 \xrightarrow{u} p \xrightarrow{e} q\xrightarrow{v} 1$ is a cycle of odd length, which is absurd. Similarly, if there is an edge  $p \xrightarrow{e} q$ for $p,q\in O$ and taking geodesic paths $u,v$ from $1$ to $p$ and from $q$ to $1$, respectively, 	$1 \xrightarrow{u} p \xrightarrow{e} q\xrightarrow{v} 1$ is a cycle of odd length. So $2\Rightarrow 3$.
	
	Now we prove that $3\Rightarrow 1$. Suppose that $\mc G$ is bipartite. Then the set of vertices of $\mc G$ can be partitioned in two sets $U$ and $V$ such that every path alternates between vertices in $U$ and vertices in $V$. Hence, a path starting in $U$ is in $U$ after an even number of steps and in $V$ after an odd number of steps. So, loops starting in the basepoint must have even length, so every element in $H$ must have even length. 
\end{proof}

We consider non-backtracking random walks in the Stallings graph $\mc G(V,E)$ of a finite index subgroup $H$ of a nonabelian free group. Notice that the Stallings graph of $H$ is saturated, in the sense that, for every vertex $v$ and generator $x$, there is an  edge labelled by $x$ with endpoint $v$ (the edge can be either incoming our outgoing). So, any reduced word can be read from any vertex. Elements in $H$ are the reduced words labelling closed loops starting (and ending) in the basepoint $1$. Hence, letting $L$ be the language of reduced words representing elements of $H$, the probability that, starting in the basepoint, we stay in the basepoint after $n$ steps is precisely $\frac{|L\cap \Sigma^n|}{|\Sigma^n|}.$ This random walk defines a Markov chain with  state space to be the set of edges. The period of a state is  defined as the greatest common divisor of the length of all loops starting (and ending) in that given state $e$. If all states have period $1$, the random walk is said to be aperiodic.

\begin{theorem}\label{period 1 2}
	Let $H\leq F_n$ be a finite index subgroup of a nonabelian free group and $\mc G$ be its Stallings graph. The non-backtracking random walk on $\mc G$ is aperiodic if and only if $H$ has a word of odd length, that is if $\mc G$ is not bipartite. If not, every edge has period $2$.
\end{theorem}
\begin{proof}
	Let $X=\{x_1,\ldots, x_n\}$ be the basis of the free group. Suppose that there is some vertex $v$ such that all loops starting in $v$ have a multiple of some $k\geq 3$ as their length and $z$ be a reduced word labelling a path from $v$ to the basepoint.
	
	Let $x_i$ be such that the product $zx_iz^{-1}$ is reduced. 	
	Since $H$ has finite index, for every $g\in F_n$, there is some $k\in \N$ such that $g^k\in H$. Take $m$ such that $x_i^m\in H$. Also, $x_i^{2m}\in H$, so $zx_i^mz^{-1}$ and $zx_i^{2m}z^{-1}$ label closed paths starting in $v$. Since the products are reduced, then $|zx_i^m z\inv|=2|z|+m$ and $|zx_i^{2m}z\inv|=2|z|+2m.$
	It follows that  
	\begin{align*}
		\begin{cases}
			&2|z|+m\equiv_k 0\\
			&  2|z|+2m\equiv_k 0.
		\end{cases}
	\end{align*}
	Hence, $m\equiv_k 0$.

	Let $x_j^\varepsilon$ be the last letter of $z$, with $\varepsilon\in \{-1,1\}$ and $n$ be such that $(x_i^{-(m-1)}x_j^{-\varepsilon})^n\in H.$ So, we have that the product $(x_i^m)(x_i^{-(m-1)}x_j^{-\varepsilon})^n=(x_ix_j^{-\varepsilon})(x_i^{-(m-1)}x_j)^{n-1}$ also belongs to $H$. Notice that $z(x_i^m)(x_i^{-(m-1)}x_j^{-\varepsilon})^n z\inv$ is a reduced word labelling a path starting in $v$. Hence,
	\begin{align*}
		2|z|+2+m(n-1)\equiv_k 0. 
	\end{align*}
	But $2|z|\equiv_k -m \equiv_k 0$, so we get that $2\equiv_k 0$, from where it follows that $k\in\{1,2\}$. So the period of any state is either $1$ or $2$. Also, we cannot have both situations happening simultaneously. Indeed, if there are loops of odd length starting in $p$ and all loopes starting in a vertex $q$ have even length, letting $z$ label a path from $q$ to $p$ and $x$ label a closed path of odd length starting in $p$. Then $zxz\inv$ labels a path of odd length starting in $q$, which is absurd.

	If $\mc G$ is not bipartite, then, by Proposition \ref{even bipartite}, there is a cycle of odd length. Hence, not all edges have period $2$, from where it follows that $\mc G$ is aperiodic. If $\mc G$ is bipartite, then all cycles have even length, from where it follows, together with the fact that every edge has period $1$ or $2$, that the period of each edge is $2$. 
\end{proof}

Defining the convention that  $\frac 1 \infty=0$, we can now prove that, taking the density on spheres, the average between the supremum density and the infimum density of a subgroup is always $\frac{1}{[G:H]}$.
\begin{theorem}\label{average finite index}
	Let $L$ be the language of reduced words representing elements in a subgroup $H$ of a nonabelian free group $F_n$. Then $$\frac{\oD(L\mid\Red(X))+D_{\inf}(L\mid \Red(X))}{2}=\frac{1}{[F_n:H]}.$$ Moreover, $D(L\mid \Red(X))$ exists if and only if $H$ has infinite index or has at least a word of odd length.
\end{theorem}
\begin{proof}
	If $H$ has infinite index, then, by Corollary \ref{cor: fi}, $\oDc(L\mid\Red(X))=\delta(H)=0$, so, by Proposition \ref{null cumulative}, $\oD(L\mid\Red(X))=0$. Thus, $D_{\inf}(L\mid \Red(X))=0$, and the equality holds.
	
	If $H$ has finite index, by Theorem \ref{period 1 2}, the non-bactracking random walk on its Stallings graph $\mc G=(V,E)$ defines a Markov chain of period $1$ or $2$. Also, the graph is connected, so the Markov chain is irreducible. 
	
	If it is aperiodic, that is $H$ has a word of odd length, then it converges to its stationary distribution \cite[Theorem 1.8.3]{[Nor97]}. So, in this case, $$\oD(L\mid\Red(X))=D_{\inf}(L\mid \Red(X))=\frac{1}{|V|}=\frac{1}{[F_n:H]},$$
	and the equality holds.
	
	If it is not, then the graph is bipartite, so after an odd number of steps, we are never closing a loop. That means that, if $k$ is odd, then $\frac{|L\cap \Sigma^k|}{|\Sigma^k|}=0$, so $$D_{\inf}(L\mid \Red(X))=\liminf_{k\to \infty}\frac{|L\cap \Sigma^k|}{|\Sigma^k|}=0.$$

	Now, our Markov chain has period 2, so we will use the convergence theorem for periodic Markov chains \cite[Theorem 1.8.5]{[Nor97]}, which yields that, for all $k\in \N$, 
	$$\frac{|\Sigma^{2k}\cap L|}{|\Sigma^{2k}|}=\frac{2}{m},$$ where $m$ is the expected return time to the basepoint. The expected return time of each edge ending in the basepoint is the inverse of its stationary distribution \cite[Proposition 1.14]{[LPW09]}, so $\frac{1}{|E|}=\frac{1}{n|V|}=\frac{1}{n[G:H]}$. Summing over all $n$ edges ending in the basepoint, we have that the expected return time to the basepoint is $\frac{1}{[G:H]}$.
	So, we deduce that 
	$$\frac{|\Sigma^{2k}\cap L|}{|\Sigma^{2k}|}=\frac{2}{[G:H]},$$
	for all $k\in \N$.
	Hence, $$\oD(L\mid \Red(X))=\frac{2}{[G:H]},$$
	and we obtain the desired equality.
\end{proof}

\subsubsection{Weak density}
In \cite{[BPR09]}, the authors give another (weaker) definition of density as the average limit $$d(L)=\lim_{n\to\infty}\frac 1 n \sum_{i=0}^{n-1}\frac{|\Sigma^i\cap L|}{|\Sigma^i|}.$$
This definition of density is more robust than the previous one. On the one hand, if density exists in the usual sense, then it also exists in the weaker sense and it has the same value \cite[Corollary 1.5]{[Mur09]}, but, for example, in \cite[Section 3.3]{[BGONBPP24]} there is an example in the Fibonacci shift where standard density does not converge, while this weaker definition does.

From our previous work, it follows that this notion of density always exists for finite index subgroups of nonabelian free groups and it is indeed $\frac{1}{[G:H]}$.

Define the weak density of a subset $K$ of a group as $G$ as 
$$\Delta(K)=\lim_{n\to \infty}\frac{1}{n}\sum_{i=0}^{n-1}\frac{|B(i)\cap K|}{|B(i)|}.$$
We will define the weak  density of a language $L_1$ relative to a larger language $L_2$ as 
$$d(L_1\mid L_2)=\lim_{n\to \infty}\frac{1}{n}\sum_{i=0}^{n-1}\frac{|\Sigma^i\cap L_1|}{|\Sigma^i\cap L_2|}.$$
As above, if $K$ is a subset of a free group and $L$ is the language of reduced words representing elements in $K$, we have that $d(L\mid \Red(X))=\Delta(K)$.

\begin{theorem}\label{weak density}
	Let $H$ be a subgroup of a nonabelian free group $F_k$. Then 
	\begin{align*}
\Delta(H) = 
\begin{cases}
	0               & \text{if } H \text{ has infinite index} \\
	\frac{1}{[F_k:H]} & \text{if } H \text{ has finite index}
\end{cases}
	\end{align*}
\end{theorem}
\begin{proof}
	If $H$ has infinite index,  then, by Corollary \ref{cor: fi}, $\oDc(L\mid\Red(X))=\delta(H)=0$, so, by Proposition \ref{null cumulative}, $\oD(L\mid\Red(X))=0$. Thus, $D_{\inf}(L\mid \Red(X))=0$, and there is convergence in the strong sense (and so in the weak sense).
	
	If $H$ has finite index, then either it has a word of odd length, in which case, by the proof of Theorem \ref{average finite index}, there is convergence in the strong sense, and so there is also convergence in the weak sense or it is contained in $F^{even}$, in which case there is no strong convergence.
	
	So, assume that $H$ has finite index and it is contained in $F^{even}$, and let $L$ be the language of reduced words representing elements in $H$. We have that $\frac{|\Sigma^i\cap L|}{|\Sigma^i\cap \Red(X)|}=0$ if $i$ is odd and the sequence $\left(\frac{|\Sigma^{2i}\cap L|}{|\Sigma^{2i}\cap \Red(X)|}\right)_{i\in \N}$ converges to $\frac{2}{[F_k:H]}$ by the proof of Theorem \ref{average finite index}.

Let $a_i = \frac{|\Sigma^i \cap L|}{|\Sigma^i \cap \Red(X)|}$. We wish to compute $\Delta(H) = \lim_{n\to\infty} \frac{1}{n} \sum_{i=0}^{n-1} a_i$. Consider the partial means for an even number of terms $n=2m$:
\begin{align*}
	\frac{1}{2m} \sum_{i=0}^{2m-1} a_i = \frac{1}{2m} \left( \sum_{j=0}^{m-1} a_{2j} + \sum_{j=0}^{m-1} a_{2j+1} \right) = \frac{1}{2m} \sum_{j=0}^{m-1} a_{2j} + 0 = \frac{1}{2} \left( \frac{1}{m} \sum_{j=0}^{m-1} a_{2j} \right).
	\end{align*}
	Since $a_{2j} \to \frac{2}{[F_k:H]}$, by \cite[Corollary 1.5]{[Mur09]}, its average also converges to the same limit.
	 Thus,$$\lim_{m\to\infty} \frac{1}{2m} \sum_{i=0}^{2m-1} a_i = \frac{1}{2} \cdot \frac{2}{[F_k:H]} = \frac{1}{[F_k:H]}.$$For the case of an odd number of terms $n=2m+1$, we have:$$\frac{1}{2m+1} \sum_{i=0}^{2m} a_i = \frac{2m}{2m+1} \left( \frac{1}{2m} \sum_{i=0}^{2m-1} a_i \right) + \frac{a_{2m}}{2m+1}.$$As $m \to \infty$, the factor $\frac{2m}{2m+1}$ tends to $1$, and since the sequence $(a_i)$ is bounded, the term $\frac{a_{2m}}{2m+1}$ tends to $0$. Therefore, the limit is the same. We conclude that $\Delta(H) = \frac{1}{[F_k:H]}$.\end{proof}

 \section{Further work}
 
 We now present two questions arising from this work.

The first question concerns Theorem \ref{2cosets} (or even Lemma \ref{dod}) for the  $\delta_{\inf}$ case. Concretely, we ask the following question:

\begin{question}
	What are the rational subsets $S$ of a nonabelian free group such that $\delta_{\inf}(S)>0$? Is there a rational subset $S$ such that $\delta_{\inf}(S)=0$ but $\delta(S)>0$?
\end{question}

Also, natural generalizations to other classes of groups (for example hyperbolic or automatic groups) might be worth considering. For instance, some automatic structure with uniqueness can represent the language which densities are calculated relative to. Also, there are Stallings automata for finite index subgroups of automatic groups with similar properties as the Stallings automata for finite index subgroups of the free group (see \cite[Theorem 6.12]{[KMW17]}).

\begin{question}
	Is it possible to obtain similar results for other classes of groups, such as hyperbolic or automatic groups?
\end{question}

 \section{Acknowledgements}
The author thanks Corentin Bodart for many suggestions on a previous draft of this paper. Particularly, the author is grateful for the connection with non-backtracking random walks and symbolic dynamics that led to Theorems \ref{thm:shifts} and \ref{average finite index}. The author is also grateful to Herman Goulet-Ouellet for many comments and references, which very significantly improved the paper.
This work research was supported by national funds through the Fundação para
 a Ciência e Tecnologia, FCT, under the project UID/04674/2025. 

\bibliographystyle{plain}
\bibliography{Bibliografia}

\clearpage

 \end{document}